\documentclass{article}
\usepackage{amsmath}
\usepackage{amsthm}
\usepackage{mathrsfs}
\usepackage{latexsym}
\usepackage{amssymb}
\usepackage{amscd}
\usepackage[dvips]{graphics}
\usepackage[all,cmtip]{xy}

\DeclareSymbolFont{AMSb}{U}{msb}{m}{n}
\DeclareMathSymbol{\N}{\mathbin}{AMSb}{"4E}
\DeclareMathSymbol{\Z}{\mathbin}{AMSb}{"5A}
\DeclareMathSymbol{\R}{\mathbin}{AMSb}{"52}
\DeclareMathSymbol{\Q}{\mathbin}{AMSb}{"51}
\DeclareMathSymbol{\I}{\mathbin}{AMSb}{"49}
\DeclareMathSymbol{\C}{\mathbin}{AMSb}{"43}

\newcommand{\dbl}{[\hspace{-0.2ex}[}
\newcommand{\dbr}{]\hspace{-0.2ex}]}
\newcommand{\db}[1]{\dbl {#1} \dbr}
\newcommand{\res}[1]{\hspace{-0.6mm}\downarrow_{\hspace{-0.25mm}{#1}}}
\newcommand{\ind}[1]{\hspace{-0.6mm}\uparrow^{\hspace{-0.25mm}{#1}}}
\newcommand{\ctens}{\widehat{\otimes}}

\newcommand{\iso}{\cong}

\newcommand{\ds}{\raisebox{0.5pt}{\,\big|\,}}

\newcommand{\X}{\mathfrak{X}}
\newcommand{\Y}{\mathfrak{Y}}
\newcommand{\norm}[1]{\textnormal{N}_{#1}}

\numberwithin{equation}{section}

\title{Green Correspondence for Virtually Pro-$p$ Groups}

\begin{document}

\newtheorem{defn}[equation]{Def{i}nition}
\newtheorem{prop}[equation]{Proposition}
\newtheorem{lemma}[equation]{Lemma}
\newtheorem{theorem}[equation]{Theorem}
\newtheorem{corol}[equation]{Corollary}
\newtheorem{question}[equation]{Questions}

\maketitle

\section{Introduction}
Let $p$ be a prime number, $G$ a finite group, $Q$ a $p$-subgroup of
$G$ and $L$ any subgroup of $G$ containing the normalizer
$\norm{G}(Q)$ of $Q$ in $G$. Let $k$ be a field of positive
characteristic $p$. In \cite{greencorrespondencepaper} J.A. Green
demonstrates a fundamental correspondence between finitely generated
$kG$-modules with vertex $Q$ and finitely generated $kL$-modules
with vertex $Q$. When $L=\norm{G}(Q)$ the Green correspondence
allows for the reduction of many questions about general modules to
questions about modules with a normal vertex.

Now let $G$ be a profinite group and $k$ a finite field of
characteristic $p$.  In \cite{profiniterelprojarxiv} we took some first
steps towards a modular representation theory of profinite groups.
In particular we demonstrated a classification theorem for
relatively projective finitely generated $k\db{G}$-modules,
introduced vertices and sources, and showed that the expected
uniqueness properties hold for these objects (under additional
hypotheses in the case of sources).  Here we generalize the Green
correspondence (properly interpreted) to the class of virtually
pro-$p$ groups.  We will reference \cite{profiniterelprojarxiv}
frequently, since many necessary foundational results are discussed
therein.

Our main result is the following:

\begin{theorem}\label{GCprofinite}
Let $G$ be a virtually pro-$p$ group, let $Q$ be a closed pro-$p$
subgroup of $G$ and let $L$ be a closed subgroup of $G$ containing
$\norm{G}(Q)$.  Let $S$ be a finitely generated indecomposable
profinite $k\db{Q}$-module with vertex $Q$.  Then there is a canonical
bijection between the set of isomorphism classes of indecomposable
profinite $k\db{L}$-modules with vertex $Q$ and source $S$, and the
set of isomorphism classes of indecomposable profinite
$k\db{G}$-modules with vertex $Q$ and source $S$.
%

More explicitly, if $V$ is an indecomposable $k\db{L}$-module with vertex $Q$
and source $S$, then the correspondent of $V$ under the above bijection is the unique
indecomposable summand of $V\ind{G}$ having vertex $Q$.
\end{theorem}

We approach the proof in two main steps.  We first demonstrate a
correspondence which is word-for-word analogous to the finite case
under the additional assumption that $L$ is open in $G$.  Using this
special case we then demonstrate the truth of the above theorem.
First let us establish some notation to be assumed throughout our
discussion.

The main concepts mentioned in this paragraph are introduced and
discussed in \cite{profiniterelprojarxiv}.  Let $G$ be a virtually
pro-$p$ group and $k$ a finite field of characteristic $p$.  All
modules are assumed to be profinite left modules.  If $U$ is a
$k\db{G}$-module and $N$ is a closed normal subgroup of $G$, then we
denote by $U_N$ the coinvariant module $k\ctens_{k\db{N}}U$ - note that this
is not a restriction.  If $U$
is finitely generated and $N$ is open in $G$, then $U_N$ is finite.
If $U$ is non-zero, finitely generated and indecomposable, then by
\cite[2.8, 2.9]{profiniterelprojarxiv} we can choose a cofinal inverse
system of open normal pro-$p$ subgroups of $G$ for which each $U_N$
is non-zero and indecomposable.  As usual, if $H$ is a closed subgroup of
$G$ and $V$ is a $k\db{H}$-module, then we denote by $V\ind{G}$ the $k\db{G}$-module
induced from $V$.  If $U$ is a $k\db{G}$-module, then we denote by $U\res{H}$ the
$k\db{H}$-module obtained by restricting the coefficients of $U$.

Let $Q$ be a closed pro-$p$ subgroup of $G$ and let $L$ be any
closed subgroup of $G$ containing $\norm{G}(Q)$.  We define the
following two sets of subgroups of $G$:
$$\X= \{X\leq_C G\,|\, X\leq xQx^{-1}\cap Q, x\notin L\}$$
$$\Y= \{Y\leq_C G\,|\, Y\leq xQx^{-1}\cap L, x\notin L\}.$$
If $\mathfrak{H}$ is a collection of subgroups of $G$, then we say a
finitely generated $k\db{G}$-module $U$ is  relatively
$\mathfrak{H}$-projective if each indecomposable summand of $U$ is
projective relative to an element of $\mathfrak{H}$.  As in the
finite case we note that $\X$ consists of proper subgroups of $Q$,
while $\Y$ may contain a conjugate of $Q$.

\section{The case where $L$ is open}\label{open L GC}

Essentially following the treatment in \cite[3.12]{benson} we prove
three lemmas which constitute the bulk of the work for the case of
open $L$.

\begin{lemma}\label{summandsofVupanddown}
Let $V$ be a finitely generated indecomposable $Q$-projective
$k\db{L}$-module. Then $V\ind{G}\res{L}\iso V\oplus V_1$, where
$V_1$ is $\Y$-projective.
\end{lemma}

\begin{proof}
By the Mackey decomposition formula \cite[2.2]{symondsdoublecoset}
we have
$$V\ind{G}\res{L}\iso \bigoplus_{x\in L\backslash G/L}x(V)\res{xLx^{-1}\cap L}\ind{L}=V\oplus\bigoplus_{x\in L\backslash G/L, x\notin L}x(V)\res{xLx^{-1}\cap L}\ind{L}$$
so we need only show that a summand of the form
$x(V)\res{xLx^{-1}\cap L}\ind{L}$ with $x\notin L$ is
$\Y$-projective.  By \cite[5.2]{profiniterelprojarxiv} the module $x(V)$
is projective relative to $xQx^{-1}$ so by
\cite[3.7]{profiniterelprojarxiv} we can choose a $k\db{xQx^{-1}}$-module
$S$ such that $x(V)\ds S\ind{xLx^{-1}}$.  Then
$$x(V)\res{xLx^{-1}\cap L}\ind{L}\ds S\ind{xLx^{-1}}\res{xLx^{-1}\cap L}\ind{L}\iso\bigoplus_y y(S)\res{(yx)Q(yx)^{-1}\cap xLx^{-1}\cap L}\ind{L}$$
where $y$ runs through a set of double coset representatives of
$$(xLx^{-1}\cap L)\backslash xLx^{-1}/xQx^{-1}.$$  Note that $yx=xlx^{-1}x=xl$ for some $l\in L$ and
$x\notin L$ implies $xl\notin L$ so that each $((yx)Q(yx)^{-1}\cap
xLx^{-1}\cap L)\in \Y$.  Hence the module $x(V)\res{xLx^{-1}\cap
L}\ind{L}$ is relatively $\Y$-projective as required.
\end{proof}

\begin{lemma}\label{summandsofVinduced}
Let $V$ be a finitely generated indecomposable $Q$-projective
$k\db{L}$-module. Then $V\ind{G}\iso U\oplus U_1$, where $U$ is
indecomposable, $V\ds U\res{L}$, and $U_1$ is $\X$-projective.
\end{lemma}

\begin{proof}
Since $V\ds V\ind{G}\res{L}$, by the Krull--Schmidt theorem
\cite[2.1]{symondspermcom2} there is an indecomposable summand $U$
of $V\ind{G}$ with $V\ds U\res{L}$. Write $V\ind{G}\iso U\oplus U_1$
and take an indecomposable summand $U'$ of $U_1$.  We wish to show
that $U'$ is relatively $\X$-projective.  Note that $U'$ is
relatively $Q$-projective.

Since $\norm{G}(Q)\leq L$ and $L$ is open, a standard compactness
argument allows us to consider a cofinal inverse system of open
normal subgroups $N$ of $G$ such that $\norm{G}(QN)\leq L$.  Fix
some $N$ in our system.  The module $U'$ is projective relative to
$QN$, so $U'\ds U'\res{QN}\ind{G}$ by \cite[3.7]{profiniterelprojarxiv}.
Since $U'\res{QN}$ is finitely generated, we can find some
indecomposable $k\db{QN}$-module $S$ such that $S\ds U'\res{QN}$ and
$U'\ds S\ind{G}$.  Now $U'\res{QN}\iso U'\res{L}\res{QN}$ so there
is an indecomposable finitely generated $k\db{L}$-module $V'$ such
that $V'\ds U'\res{L}$ and $S\ds V'\res{QN}$.  Note that $V'$ is a
direct summand of $V\ind{G}\res{L}$ distinct from $V$, so by
Lemma \ref{summandsofVupanddown} it is projective relative to a subgroup
of the form $tQt^{-1}\cap L$ with $t\in G, t\notin L$.  Let $T$ be a
$k\db{tQt^{-1}\cap L}$-module such that $V'\ds T\ind{L}$.  From the
Mackey decomposition theorem \cite[2.2]{symondsdoublecoset} we have
$$S\ds V'\res{QN}\ds T\ind{L}\res{QN} \iso \bigoplus_{l\in QN\backslash L/tQt^{-1}\cap L}l(T)\res{(lt)Q(lt)^{-1}\cap
                                    QN}\ind{QN}.$$
Since $t\notin L$ it follows that $S$ is projective relative to a
subgroup of the form $xQx^{-1}\cap QN$ for some $x\notin L$.  Since
$U'\ds S\ind{G}$ we have shown that for each $N$ in our system the
module $U'$ is projective relative to a subgroup of the form
$xQx^{-1}\cap QN$ for some $x\notin L$ that depends on $N$.

We would like to find some $x\in G, x\notin L$ for which $U'$ is
projective relative to $xQx^{-1}\cap QN$ for every $N$ in our
system. Denote by $C_N$ the non-empty set of $x\in G, x\notin L$ for
which $U'$ is relatively $[xQx^{-1}\cap QN]$-projective.
If ever $N\leq M$ and $x\in C_N$ then certainly $U'$ is projective
relative to $xQx^{-1}\cap QM$, so if $x\in C_N$ then $x\in C_M$.
Since each $C_N$ is closed in $G$ the standard compactness argument
now shows that $\bigcap_N C_N\neq\emptyset$.

Choose some $x\in G, x\notin L$ for which $U'$ is projective
relative to $xQx^{-1}\cap QN$ for each $N$ in our system.  By
\cite[4.2]{profiniterelprojarxiv} it follows that $U'$ is projective
relative to
$$\bigcap_N(xQx^{-1}\cap QN)=xQx^{-1}\cap(\bigcap_N QN)=xQx^{-1}\cap Q$$
as required.
\end{proof}

In the finite case the following lemma is an easy corollary of
Lemma \ref{summandsofVupanddown}.  In our more general context it requires
a little more care.

\begin{lemma}\label{appropriate summand of UresL}
Let $U$ be a finitely generated indecomposable $k\db{G}$-module with
vertex $Q$. There is a finitely generated indecomposable
$k\db{L}$-module $V$ with vertex $Q$ such that $U\ds V\ind{G}$ and
$V\ds U\res{L}$.
\end{lemma}

\begin{proof}
We work in a cofinal inverse system of $N\lhd_O G$ with
$\norm{G}(QN)\leq L$. We first show that $U\res{L}$ has an
indecomposable summand with vertex $Q$. Since $U\ds U\res{L}\ind{G}$
we have $U\res{L}$ has at least one summand with vertex conjugate to
$Q$ in $G$.  Let $\mathcal{X}$ denote the non-empty set of
isomorphism classes of $V\ds U\res{L}$ having vertex conjugate to
$Q$. For each $N$, the fact that $U\ds U\res{QN}\ind{G}$ implies that $U\ds
V\res{QN}^L\ind{G}$ for some $V\in \mathcal{X}$, and so $U\ds
W\ind{G}$ for some $W\ds V\res{QN}$. Clearly $W$ has vertex
$yQy^{-1}\subseteq QN$.

Suppose $V$ has vertex $xQx^{-1}$ and let $S$ be a
$k\db{xQx^{-1}}$-module with $V\ds S\ind{L}$. Applying Mackey's
formula to $W\ds S\ind{L}\res{QN}$ it follows that $V$ has vertex
$L$-conjugate to a subgroup of $QN$, and so $V$ has a vertex
contained in $QN$.  Note that $\mathcal{X}$ is a finite set, so some
element of $\mathcal{X}$ must have vertex contained in $QN$ for a
cofinal subset of $N\lhd_O G$ and hence some element of
$\mathcal{X}$ has vertex $Q$.

Let $Z$ be an indecomposable summand of $U\res{L}$ and suppose that
for all $N$ in our system there is some $x\notin L$ such that $Z$ is
$xQNx^{-1}$-projective. Denote by $C_N$ the non-empty set of all
such $x\notin L$.  If $x\in C_N$ then $xq\in C_N$ for all $q\in QN$
and so each set $C_N$ is closed in $G$.

If $N\leq M$ and $x\in C_N$ then certainly $x\in C_M$.  By
compactness we now have that $\bigcap_N C_N\neq\emptyset$.  Fix
$x\in \bigcap_N C_N$.  It follows that $Z$ is $xQx^{-1}N$-projective
for each $N$, and so $Z$ is $xQx^{-1}$-projective.  Note that for
any $l\in L$ we have $(lx)Q(lx)^{-1}\neq Q$ since $lx\notin L$. From
the conjugacy of vertices \cite[4.6]{profiniterelprojarxiv} it follows
that $Z$ does not have vertex $Q$.

Since there is an indecomposable summand of $U\res{L}$ with vertex
$Q$ the contrapositive of the previous argument shows there is some
$N_0\lhd_O G$ such that this summand is not projective relative to
$xQN_0x^{-1}$ for any $x\notin L$.  From now on we work within the
cofinal system of $N\lhd_O G$ with $N\leq N_0$.

Let $\mathcal{T}$ denote the (finite, non-empty) set of isomorphism
classes of indecomposable $V\ds U\res{L}$ such that $U\ds V\ind{G}$.
We wish to find an element of $\mathcal{T}$ with vertex $Q$.  Choose
$N$ in our system.  Since $U\ds U\res{QN}\ind{G}$ we take some
indecomposable summand $V\ds U\res{QN}\ind{L}$ such that $U\ds
V\ind{G}$.  Since $V\ds V\ind{G}\res{L}$, by
Lemma \ref{summandsofVupanddown} we have two possibilities:
\begin{itemize}
\item $V\ds U\res{L}$ or
\item Each summand of $U\res{L}$ is projective relative to $xQNx^{-1}\cap
L$ for some $x\notin L$.
\end{itemize}
By our choice of $N$ the latter cannot happen, so that $V\ds
U\res{L}$ and so $V\in \mathcal{T}$.  Thus for all $N$ there is an
element of $\mathcal{T}$ which is $QN$-projective, and so there is
an element of $\mathcal{T}$ which has vertex $Q$, as required.
\end{proof}

\begin{prop}\label{GCforLopen}
Let $G$ be a virtually pro-$p$ group, $Q$ a closed pro-$p$ subgroup
of $G$ and let $L$ be an open subgroup of $G$ containing
$\textnormal{N}_G(Q)$.  Then we have the following correspondence
between finitely generated indecomposable $k\db{G}$-modules with
vertex $Q$, and finitely generated indecomposable $k\db{L}$-modules
with vertex $Q$:
\begin{enumerate}
\item If $U$ is a finitely generated indecomposable $k\db{G}$-module with vertex $Q$,
then there is a unique indecomposable summand $f(U)$ of $U\res{L}$
with vertex $Q$, and the rest have vertex in $\Y$.
\item If $V$ is a finitely generated indecomposable $k\db{L}$-module with vertex $Q$,
then there is a unique indecomposable summand $g(V)$ of $V\ind{G}$
with vertex $Q$, and the rest have vertex in $\X$.
\item The given correspondence is one-one in the sense that $f(g(V))\iso
V$ and \newline $g(f(U))\iso U$.
\end{enumerate}
\end{prop}

\begin{proof}
\begin{enumerate}
\item By Lemma \ref{appropriate summand of UresL} we have that $U\ds V\ind{G}$ for some
finitely generated indecomposable $k\db{L}$-module $V$ with vertex
$Q$. Thus $U\res{L}\ds V\ind{G}\res{L}$.  By
Lemma \ref{summandsofVupanddown}, $V$ is the only summand of
$V\ind{G}\res{L}$ with vertex $Q$ and the rest have vertex in $\Y$,
so that $U\res{L}$ has at most one summand with vertex $Q$.  On the
other hand, again by Lemma \ref{appropriate summand of UresL} we have that
$U\res{L}$ has at least one summand with vertex $Q$. Hence we set
$f(U)=V$ and the claim holds.
\item We have $V\ds V\ind{G}\res{L}$ so we choose an indecomposable
summand $U\ds V\ind{G}$ such that $V\ds U\res{L}$.  By
Lemma \ref{summandsofVinduced}, we have $V\ind{G}\iso U\oplus U_1$ where
$U_1$ is $\X$-projective.  The module $U$ has vertex $Q$ since if it
had smaller vertex then the Mackey decomposition theorem shows that
$V$ would as well.  Thus, we take $g(V)=U$ and we are done.
\item This is clear.
\end{enumerate}
\end{proof}

\section{A more general case}\label{general L GC}

We retain the notation from above but drop the assumption that $L$
is open in $G$.  When $L$ was open and $U,V$ were Green
correspondents as above, we see in particular that $V\ds U\res{L}$.
This need not be the case when $L$ has infinite index in $G$ - an
example of this phenomenon can be found in the last section of
\cite{symondsdoublecoset}. For this reason we now focus on the map
$g$.

Let $V$ be an indecomposable finitely generated $k\db{L}$-module
with vertex $Q$.  By \cite[5.1]{profiniterelprojarxiv} we can choose a
cofinal inverse system of $N\lhd_O G$ for which $V\ind{LN}$ is
indecomposable.  We work in this system as we prove the following
key lemma:

\begin{lemma}\label{induced module has same vertex}
For any given $M\lhd_O G$ in our inverse system the module
$V\ind{LM}$ has vertex $Q$.
\end{lemma}

\begin{proof}
Certainly $V\ind{LM}$ is relatively $Q$-projective, so we choose
some vertex $R$ of $V\ind{LM}$ contained in $Q$.  We will show that
$V$ is $R$-projective. Consider the cofinal inverse system of those
$N\lhd_O G$ contained inside $M$, noting that for each such $N$ the
module $V\ind{LN}$ is indecomposable.

Let $S$ be a $k\db{R}$-module such that $V\ind{LM}\ds S\ind{LM}$.
Then for each $N\leq M$ we have
$$V\ind{LN}\ds V\ind{LM}\res{LN}\ds S\ind{LM}_R\res{LN}\iso \bigoplus_{x\in LN\backslash LM/R}x(S)\res{xRx^{-1}\cap LN}\ind{LN}$$
so that $V\ind{LN}$ is $xRx^{-1}$-projective for some $x\in LM$ and
hence has vertex $xRx^{-1}$. Denote by $C_N$ the non-empty set of
all $y\in LM$ with the property that $yRy^{-1}$ is a vertex of
$V\ind{LN}$. Then $C_N$ is a finite union of right cosets of $LN$ so
is a closed subset of $LM$. We would like to show that $\bigcap_N
C_N\neq\emptyset$.

Given $N_1,\hdots,N_n$, let $N'=N_1\cap\hdots\cap N_n$.  Then by the
argument above $C_{N'}\neq \emptyset$. But $C_{N'}\subseteq C_{N_i}$
for each $i$, since if $V\ind{LN'}$ is induced from a
$yRy^{-1}$-module, then so is each $V\ind{LN_i}$. Thus,
$\emptyset\neq C_{N'}\subseteq C_{N_1}\cap \hdots \cap C_{N_n}$ and
so by compactness $\bigcap_N C_N\neq\emptyset$. It follows that we
can find some $y\in LM$ so that $V\ind{LN}$ is $yRy^{-1}$-projective
for each $N\leq M$.

We move now from induced modules to coinvariant modules.  Note that
if $V\ind{LN}$ is $yRy^{-1}$-projective then it is certainly
$yRNy^{-1}$-projective, so for some $yRNy^{-1}$-module $T$ we have
$V\ind{LN}\ds T\ind{LN}$.  Now
$$V_{L\cap N}\iso (V\ind{LN})_N\ds (T\ind{LN})_N\iso T_N\ind{LN}$$
by \cite[2.6]{profiniterelprojarxiv} so that $V_{L\cap N}$ is
$yRNy^{-1}$-projective for each $N$ in our system.  Now by
\cite[3.5]{profiniterelprojarxiv} the module $V$ is $yRy^{-1}$-projective
and so some conjugate of $yRy^{-1}$ contains $Q$.  Thus $R\leq Q\leq
zRz^{-1}$ for some $z\in LM$, so $R=Q$ and we are done.
\end{proof}

Recall that $L$ contains the normalizer of $Q$ in $G$.

\begin{corol}\label{GCup}
Let $V$ be an indecomposable finitely generated $k\db{L}$-module with vertex $Q$.  Then
$V\ind{G}$ has a unique summand $g(V)$ with vertex $Q$, and the rest
have vertex in $\X$.
\end{corol}

\begin{proof}
We choose some $M\lhd_O G$ for which $V\ind{LM}$ is indecomposable.
By Lemma \ref{induced module has same vertex}, $V\ind{LM}$ has vertex $Q$.
But now by Proposition \ref{GCforLopen}, $V\ind{G}\iso V\ind{LM}\ind{G}$ has a
unique summand $g(V)$ with vertex $Q$ and the rest have vertex in
$$\{X\leq_C G\,|\, X\leq xQx^{-1}\cap Q, x\notin LM\}$$
but this is a subset of $\X$ and so we are done.
\end{proof}

We can now prove Theorem 1.1:

\begin{proof}
The map $g$ from Corollary \ref{GCup} restricted to those modules with source
$S$ has the appropriate image and domain.  We need only check that
$g$ is bijective.

First we show that if $U$ is an indecomposable $k\db{G}$-module with
vertex $Q$ and source $S$, then there is some indecomposable
$k\db{L}$-module $V$ with vertex $Q$ and source $S$ such that $U\iso
g(V)$.  But this is clear since if $S\ind{L}\iso
V_1\oplus\hdots\oplus V_n$ is a decomposition into indecomposable
summands then
$$U\ds S\ind{G}\iso S\ind{L}\ind{G}\iso V_1\ind{G}\oplus\hdots\oplus V_n\ind{G}$$
and so $U\ds V_i\ind{G}$ for some $i$ since $U$ has local
endomorphism ring by \cite[4.4]{profiniterelprojarxiv}.  Clearly $V_i$
has vertex $Q$. This shows that $g$ is surjective.

It remains to show that if $V,W$ are finitely generated
indecomposable $k\db{L}$-modules having vertex $Q$ and source $S$,
and $g(V)\iso g(W)$ as $k\db{G}$-modules, then $V\iso W$ as
$k\db{L}$-modules.  Choose a cofinal inverse system of $N\lhd_O G$
for which both $V\ind{LN}$ and $W\ind{LN}$ are indecomposable. Let
$g(V)\iso U\iso g(W)$.  The modules $V\ind{LN}$ and $W\ind{LN}$ are
both Green correspondents of $U$ in the sense of Proposition \ref{GCforLopen}
and so $V\ind{LN}\iso W\ind{LN}$ for each $N$ in our inverse system.
But
\begin{align*}
V\ind{LN} & \iso W\ind{LN} \\
\implies (V\ind{LN})_N & \iso (W\ind{LN})_N \\
\implies V_{L\cap N} & \iso W_{L\cap N}
\end{align*}
for each $N$, and so $V\iso W$ by \cite[3.4]{profiniterelprojarxiv}.
\end{proof}

\section{Acknowledgements}

The author gratefully acknowledges the help and support of his PhD
supervisor Peter Symonds throughout this exciting project.  Thanks also
to the referee for helpful comments.

\bibliographystyle{plain}
\bibliography{bibliography}

\begin{thebibliography}{1}

\bibitem{benson}
D.J. Benson.
\newblock {\em Representations and Cohomology I}.
\newblock Cambridge University Press, Cambridge, 1995.

\bibitem{greencorrespondencepaper}
J.A. Green.
\newblock A transfer theorem for modular representations.
\newblock {\em Journal Of Algebra}, 1(1):73--84, 1964.

\bibitem{profiniterelprojarxiv}
J.W. MacQuarrie.
\newblock Modular representations of profinite groups.
\newblock Preprint (2010), available at
  \texttt{http://arxiv.org/abs/1011.2899}, 2010.

\bibitem{symondspermcom2}
P.A. Symonds.
\newblock On the construction of permutation complexes for profinite groups.
\newblock {\em Geometry and Topology Monographs}, 11(1):369--378, 2007.

\bibitem{symondsdoublecoset}
P.A. Symonds.
\newblock Double coset formulas for profinite groups.
\newblock {\em Communications in Algebra}, 36(3):1059--1066, 2008.

\end{thebibliography}

\end{document}